\documentclass[a4papaer]{article}
\usepackage{amsmath, amssymb, amscd, amsthm}
\usepackage{authblk}
\usepackage{geometry}
\usepackage{color}

\newtheorem{Thm}{Theorem}
\newtheorem{Lem}{Lemma}
\newtheorem{Prop}{Proposition}
\newtheorem{Cor}{Corollary}

\theoremstyle{definition}
\newtheorem{Def}{Definition}

\theoremstyle{remark}
\newtheorem{Rem}{Remark}

\DeclareMathOperator*{\disc}{disc}

\title{Propagation of boundary-induced discontinuity in stationary radiative transfer and its application to the optical tomography}
\author[1]{I-Kun Chen \thanks{ikun.chen@gmail.com} }
\author[2]{Daisuke Kawagoe \thanks{k.goe.dai@gmail.com} }
\affil[1]{Institute of Applied Mathematical Sciences, National Taiwan University}
\affil[2]{Institute of Applied Mathematics, Inha University}

\allowdisplaybreaks

\begin{document}
\maketitle
        
\begin{abstract}
We consider a boundary value problem of the stationary transport
equation with the incoming boundary condition in two or three dimensional bounded convex domains. We discuss discontinuity of the solution to the boundary value problem arising from discontinuous incoming boundary data, which we call the boundary-induced discontinuity. In particular, we give two kinds of sufficient conditions on the incoming boundary data for the boundary-induced discontinuity. We propose a method to reconstruct attenuation coefficient from jumps in boundary measurements.

\end{abstract}

\section{Introduction}
We consider the stationary transport equation:
\begin{equation} \label{eq:STE}
\xi \cdot \nabla_x f(x, \xi) + \mu_t(x) f(x, \xi) = \mu_s(x) \int_{S^{d-1}} p(x, \xi, \xi^\prime) f(x, \xi^\prime)\,d\sigma_{\xi^\prime}.
\end{equation}
The stationary transport equation describes propagation of photons \cite{Chand}. The function $f(x, \xi)$ stands for density of photons at a point $x \in \mathbb{R}^d$, $d = 2$ or $3$, with a direction $\xi \in S^{d-1}$. Here, $S^{d-1}$ is the unit sphere in $\mathbb{R}^d$. Two coefficients $\mu_t$ and $\mu_s$ and the integral kernel $p$ characterize absorption and scattering of photons in the media; they are called the attenuation coefficient, the scattering coefficient, and the scattering phase function, respectively. 

We introduce a domain in which we consider the equation (\ref{eq:STE}). Let $\Omega$ be a bounded convex domain in $\mathbb{R}^d$ with the $C^1$ boundary $\partial \Omega$. We assume that $\overline{\Omega} = \cup_{j = 1}^N \overline{\Omega_j}$, where $\Omega_j$, $1 \leq j \leq N$, are disjoint (open) subdomains of $\Omega$ with piecewise $C^1$ boundaries. Let $\Omega_0 := \cup_{j = 1}^N \Omega_j$. We assume that, for all $(x, \xi) \in \Omega \times S^{d-1}$, the half line $\{ x - t\xi | t \geq 0 \}$ intersects with $\partial \Omega_0$ at most finite times. In other words, for all $(x, \xi) \in \Omega \times S^{d-1}$, there exist positive integer $l(x, \xi)$ and real numbers $\{ t_j(x, \xi) \}_{j = 1}^{l(x, \xi)}$ such that $0 \leq t_1(x, \xi) < t_2(x, \xi) < \cdots < t_{l(x, \xi)}(x, \xi)$, $x - t\xi \in \partial \Omega_0$ if and only if $t = t_j(x, \xi)$, and $\sup_{(x, \xi) \in \Omega \times S^{d-1}} l(x, \xi) < \infty$. This assumption is called generalized convexity condition for $\Omega_0$ \cite{1993Anik}. We consider the equation (\ref{eq:STE}) in $\Omega_0 \times S^{d-1}$. In what follows, we use these notations $t_j (x, \xi)$ and $l(x, \xi)$, and we put $t_0(x, \xi) = 0$.

We introduce a boundary value problem to the equation (\ref{eq:STE}). Denote the incoming boundary $\Gamma_-$ and the outgoing boundary $\Gamma_+$ by 
\begin{equation*}
\Gamma_{\pm} := \{ (x, \xi) \in \partial \Omega \times S^{d-1} | \pm n(x) \cdot \xi > 0 \},
\end{equation*}
where $n(x)$ is the outer unit normal vector at $x \in \partial \Omega$ and $n(x) \cdot \xi$ is the inner product of two vectors $n(x)$ and $\xi$ in $\mathbb{R}^d$. We consider the incoming boundary value problem to seek a solution $f$ to the equation (\ref{eq:STE}) satisfying
\begin{equation} \label{eq:BC}
f(x, \xi) = f_0(x, \xi), \quad (x, \xi) \in \Gamma_-
\end{equation}
for a given function $f_0$ on $\Gamma_-$.


Our aim in this paper is to propose a way to reconstruct the attenuation coefficient $\mu_t$ from the boundary data, $f_0$ and $f|_{\Gamma_+}$, of the solution $f$ to the boundary value problem (\ref{eq:STE})-(\ref{eq:BC}) with $\mu_s$ and $p$ also unknown. This reconstruction is related to the optical tomography, which is a new medical imaging technology \cite{Arr}.

Anikonov et al. \cite{1993Anik} made use of propagation of the boundary-induced discontinuity, which is the discontinuity of the solution to the problem arising from discontinuous incoming boundary data, in order to solve the inverse problem. They showed that a jump of the boundary-induced discontinuity propagates along a positive characteristic line when the boundary data has a jump with respect to direction $\xi$, and it is observed as a jump of the outgoing boundary data on a discontinuous point, which locates on the tip of the characteristic line. The exponential decay of the jump contains information about the X-ray transform of the attenuation coefficient $\mu_t$, which is defined by 
\begin{equation*}
(X \mu_t)(x, \xi) := \int_{\mathbb{R}} \mu_t (x - r\xi)\,dr, \quad (x, \xi) \in \mathbb{R}^d \times S^{d-1}.
\end{equation*}
They applied the inverse X-ray transform to the observed data in order to determine the unknown coefficient $\mu_t$ from its image $X \mu_t$ \cite{Nat}. 

On the other hand, a jump of the boundary-induced discontinuity also propagates along a positive characteristic line when the boundary data has a jump with respect to space $x$. Aoki et al. \cite{Aoki} showed this property for the case of the two dimensional half homogeneous space with an incoming boundary data independent of $\xi$. The authors \cite{Kawa} extended their result to $d$-dimensional ($d \geq 2$) inhomogeneous slab domains with incoming boundary data depending on $\xi$, although we assumed a slight condition on two coefficients. In this paper, we further extend the result in \cite{Kawa} to a bounded  convex domain case. In addition to the discontinuity with respect to direction $\xi$, which is presented in Anikonov et al \cite{1993Anik}, we also discuss the discontinuity with respect to space $x$. 

$\mu_t$, $\mu_s$, and $p$ can be reconstructed by the use of an albedo operator \cite{Bal} \cite{Chou} \cite{Wang}. The albedo operator is the operator which maps the incoming boundary data $f_0$ to the outgoing boundary data $f|_{\Gamma_+}$. However, it is not feasible to observe an albedo operator from finite times experiments.  On the contrary, our proposed method uses jumps in boundary measurements, which are observed by finite times experiments. Besides, though we can only reconstruct $\nu_t$ our approach, it is the very coeffiecent contains the most important information. 

We assume that $\mu_t$ and $\mu_s$ are nonnegative bounded functions on $\mathbb{R}^d$ such that $\mu_t$ and $\mu_s$ are continuous on $\Omega_0$, $\mu_t(x) \geq \mu_s(x)$ for $x \in \Omega_0$, $\mu_t(x) = \mu_s(x) = 0$ for $x \in \mathbb{R}^d \backslash \Omega_0$, and discontinuity may occur only at  $\partial \Omega_0$. We also assume that the integral kernel $p$ is a nonnegative bounded function on $\mathbb{R}^d \times S^{d-1} \times S^{d-1}$ which is continuous on $\Omega_0 \times S^{d-1} \times S^{d-1}$ and $p(x, \xi, \xi^\prime) = 0$ for $(x, \xi, \xi^\prime) \in (\mathbb{R}^d \backslash \Omega_0) \times S^{d-1} \times S^{d-1}$, and satisfies
\begin{equation*}
\int_{S^{d-1}} p(x, \xi, \xi^\prime)\,d\sigma_{\xi^\prime} = 1
\end{equation*}
for all $(x, \xi) \in \Omega_0 \times S^{d-1}$. We regard the directional derivative $\xi \cdot \nabla_x f(x, \xi)$ as 
\begin{equation*}
\xi \cdot \nabla_x f(x, \xi) := \left. \frac{d}{dt} f(x + t\xi, \xi) \right|_{t = 0}.
\end{equation*}
Finally, the measure $d\sigma_{\xi^\prime}$ is the Lebesgue measure on the sphere $S^{d-1}$.

We introduce some notations. Let 
\begin{equation*}
D := (\Omega \times S^{d-1}) \cup \Gamma_-, \quad \overline{D} := D \cup \Gamma_+,
\end{equation*}
and we define two functions $\tau_\pm$ on $\overline{D}$ by 
\begin{equation*}
\tau_{\pm}(x, \xi) := \inf \{ t > 0 | x \pm t\xi \not\in \Omega \}.
\end{equation*}
Let $\Gamma_{-, \xi}$ and $\Gamma_{-, x}$ be projections of $\Gamma_-$ on $\partial \Omega$ and $S^{d-1}$ respectively;
\begin{equation*}
\Gamma_{-, \xi} := \{ x \in \partial \Omega | n(x) \cdot \xi < 0 \}, \quad \xi \in S^{d-1}
\end{equation*}
and
\begin{equation*}
\Gamma_{-, x} := \{ \xi \in S^{d-1} | n(x) \cdot \xi < 0 \}, \quad x \in \partial \Omega.
\end{equation*}
Let $\disc(f)$ be a set of the discontinuous points for a function $f$.

The first main result shows how the boundary-induced discontinuity propagates in the media.
\begin{Thm} \label{thm:MR1}
Suppose that a boundary data $f_0$ is bounded and that it satisfies at least one of the following two conditions.
\begin{enumerate}
\item $f_0(x, \cdot)$ is continuous on $\Gamma_{-, x}$ for almost all $x \in \partial \Omega$,
\item $f_0(\cdot, \xi)$ is continuous on $\Gamma_{-, \xi}$ for almost all $\xi \in S^{d-1}$.
\end{enumerate}
Then, there exists a unique solution $f$ to the boundary value problem $(\ref{eq:STE})$-$(\ref{eq:BC})$, and we have
\begin{equation*} \label{eq:D}
\disc(f) = \{ (x_* + t\xi_*, \xi_*) | (x_*, \xi_*) \in \disc(f_0), 0 \leq t < \tau_+(x_*, \xi_*) \}. 
\end{equation*}
\end{Thm}

Theorem \ref{thm:MR1} shows that the boundary-induced discontinuity propagates only along a positive characteristic line starting from a discontinuous point of the incoming boundary data. Here, we call a bounded function $f$ on $D$ a solution to the boundary value problem (\ref{eq:STE})-(\ref{eq:BC}) if (i) it has the directional derivative $\xi \cdot \nabla_x f(x, \xi)$ at all $(x, \xi) \in \Omega_0 \times S^{d-1}$, (ii) it satisfies the stationary transport equation (\ref{eq:STE}) for all $(x, \xi) \in \Omega_0 \times S^{d-1}$ and the boundary condition (\ref{eq:BC}) for all $(x, \xi) \in \Gamma_-$, (iii) $f(\cdot, \xi)$ is continuous along the line $\{ x + t\xi | t \in \mathbb{R}\} \cap (\Omega \cup \Gamma_{-, \xi})$ for all $(x, \xi) \in D$, and (iv) $\xi \cdot \nabla_x f(\cdot, \xi)$ is continuous on the open line segments $\{ x + t\xi | t \in (t_{j_-1}(x, \xi), t_j(x, \xi) ) \}$, $j = 1, \ldots, l(x, \xi)$ with $t_0(x, \xi) = 0$ for all $(x, \xi) \in \Omega_0 \times S^{d-1}$. Also, a positive characteristic line from a point $(x, \xi) \in \Gamma_-$ is defined by $\{ (x + t\xi, \xi) | t \geq 0\}$.

\begin{Rem}
Theorem \ref{thm:MR1} implies that, for a bounded continuous boundary data $f_0$ on $\Gamma_-$, there exists a unique solution $f$, which is bounded continuous $D$.
\end{Rem}

\begin{Rem}
Anikonov et al. \cite{1993Anik} showed Theorem \ref{thm:MR1} with the condition 2. Our main contribution is to show Theorem \ref{thm:MR1} with the condition 2.
\end{Rem}


As the second main result, we shall discuss the boundary-induced discontinuity of the solution extended up to $\Gamma_+$. In other words, we can extend the domain of the solution $f$ up to $\Gamma_+$ and we see that the boundary-induced discontinuity propagates along a positive characteristic line up to $\Gamma_+$. 
\begin{Thm} \label{thm:MR2}
Let $f$ be the solution to the boundary value problem $(\ref{eq:STE})$-$(\ref{eq:BC})$. Then, it can be extended up to $\Gamma_+$, which is denoted by $\overline{f}$, by
\begin{equation*}
\overline{f}(x, \xi) := 
\begin{cases}
f(x, \xi), \quad &(x, \xi) \in D, \\
\displaystyle \lim_{t \downarrow 0} f(x - t\xi, \xi), \quad &(x, \xi) \in \Gamma_+.
\end{cases}
\end{equation*}
Moreover, we have
\begin{equation*}
\disc(\overline{f}) = \{ (x_* + t\xi_*, \xi_*) | (x_*, \xi_*) \in \disc(f_0), 0 \leq t \leq \tau_+(x_*, \xi_*) \}. 
\end{equation*}
\end{Thm}

We state the decay of the boundary-induced discontinuity in some situation. Let $\gamma$ be two points in $\partial \Omega$ when $d = 2$, while let $\gamma$ be a simple closed curve in $\partial \Omega$ when $d = 3$. Then, $\gamma$ splits $\partial \Omega$ into two connected components $A$ and $B$, that is $\partial \Omega = A \cup B \cup \gamma$ and $A \cap B = A \cap \gamma = B \cap \gamma = \emptyset$. We put an incoming boundary data $f_0$ by
\begin{equation} \label{eq:BCJ}
f_0(x, \xi) =
\begin{cases}
I, \quad (x, \xi) \in ( (A \cup \gamma) \times S^{d-1} ) \cap \Gamma_-,\\
0, \quad (x, \xi) \in (B \times S^{d-1} ) \cap \Gamma_-,
\end{cases}
\end{equation} 
where $I$ is a constant. We note that $f_0$ satisfies the condition 2 of Theorem \ref{thm:MR1}, and that $\disc(f_0) = \{ (x_*, \xi_*) | x_* \in \gamma, \xi_* \in \Gamma_{-, x_*} \}$.


For $(\overline{x}, \overline{\xi}) \in \disc(\overline{f})$, we define a jump $[\overline{f}] (\overline{x}, \overline{\xi})$ by
\begin{equation*}
[\overline{f}](\overline{x}, \overline{\xi}) := \lim_{\substack{x \rightarrow \overline{x}, \\ P(x, \overline{\xi}) \in (A \cup \gamma)}} f(x, \overline{\xi}) - \lim_{\substack{x \rightarrow \overline{x}, \\ P(x, \overline{\xi}) \in B}} f(x, \overline{\xi}), 
\end{equation*}
where 
\begin{equation*}
P(x, \xi) := x - \tau_-(x, \xi)\xi.
\end{equation*} 
We note that, in our situation, $[f_0](x, \xi) = I$ for all $(x, \xi) \in \disc(f_0) = (\gamma \times S^{d-1}) \cap \Gamma_-$. In this situation, we have the following theorem, which is the most important in this paper.

\begin{Thm} \label{thm:MR3}
Let $\overline{f}$ be the extended solution to the boundary value problem $(\ref{eq:STE})$-$(\ref{eq:BC})$ with the incoming boundary data given by $(\ref{eq:BCJ})$, and let $(x^*, \xi^*) \in \disc(\overline{f})$. Then,
\begin{equation*}
[\overline{f}](x^*, \xi^*) = I \exp \left( - \int_0^{\tau_-(x^*, \xi^*)} \mu_t(x^* - r\xi^*)\,dr \right).
\end{equation*}
\end{Thm}

In particular, we take a point $(x^*, \xi^*) \in \disc(\overline{f}) \cap \Gamma_+$. From Theorem \ref{thm:MR3}, we have
\begin{equation*}
X \mu_t (x^*, \xi^*) = \int_0^{\tau_-(x^*, \xi^*)} \mu_t(x^* - r\xi^*)\,dr = - \log \left( [\overline{f}](x^*, \xi^*) / I \right).
\end{equation*}
The right hand side is obtained from observed data. By arranging $\gamma$, we can observe the image $X \mu_t$ of the X-ray transform of $\mu_t$. Then, applying the well-known method in \cite{Nat}, we can reconstruct the attenuation coefficient $\mu_t$..

The ingredient of the rest part in this paper is as follows. In section 2, we derive an integral equation from the boundary value problem (\ref{eq:STE})-(\ref{eq:BC}), and we show existence and uniqueness of solutions to the derived integral equation. In section 3, we discuss regularity of the solution to the integral equation. Especially, we decompose the solution into two parts, the discontinuous part and the continuous part. In section 4, we prove that the solution to the integral equation is indeed that of the boundary value problem (\ref{eq:STE})-(\ref{eq:BC}) under the assumption in Theorem \ref{thm:MR1}. In section 5, we extend the definition domain of the solution $f$ up to the outgoing boundary $\Gamma_+$ and discuss the boundary-induced discontinuity of the extended solution. In section 6, we discuss the decay of a jump of  the boundary-induced discontinuity. In other words, we prove Theorem \ref{thm:MR3}. 

\section{Existence and uniqueness of solutions to the stationary transport equation}
We derive an integral equation from the boundary value problem (\ref{eq:STE})-(\ref{eq:BC}), and we show existence and uniqueness of a solution to the derived integral equation.

For all $(x, \xi) \in D$, integrating the equation (\ref{eq:STE}) with respect to $x$ along the line $\{x - t\xi | t > 0\}$ until the line intersects with the boundary $\partial \Omega$ and taking the boundary condition (\ref{eq:BC}) into consideration, we obtain the following integral equation: 
\begin{align} \label{eq:IE}
f(x, \xi) =& \exp \Bigl(- M_t \bigl(x, \xi; \tau_-(x, \xi) \bigr) \Bigr) f_0(P(x, \xi), \xi) \nonumber\\
&+ \int_0^{\tau_-(x, \xi)} \mu_s(x - s\xi) \exp \Bigl( - M_t (x, \xi; s) \Bigr) \int_{S^{d-1}} p(x - s\xi, \xi, \xi^\prime) f(x - s\xi, \xi^\prime)\,d\sigma_{\xi^\prime}ds,
\end{align}
where
\begin{equation*}
M_t (x, \xi; s) := \int_0^s \mu_t(x - r\xi)\,dr.
\end{equation*}
We call a bounded function $f$ on $D$ satisfying the integral equation (\ref{eq:IE}) for all $(x, \xi) \in D$ a solution to the equation (\ref{eq:IE}). We note that, although solutions to the boundary value problem (\ref{eq:STE})-(\ref{eq:BC}) satisfy the integral equation (\ref{eq:IE}), the converse does not hold in general. However, as we will see later in section 4, the solution to the integral equation (\ref{eq:IE}) is also the solution to the boundary value problem (\ref{eq:STE})-(\ref{eq:BC}) under the assumption given in Theorem 1. Therefore, we focus on discussing existence and uniqueness of a solution to the integral equation (\ref{eq:IE}).

\begin{Prop}
The solution to the integral equation $(\ref{eq:IE})$ is unique, if it exists.
\end{Prop}

\begin{proof}
Let $f_1$ and $f_2$ be two solutions to the integral equation (\ref{eq:IE}). Then the difference $\tilde{f} := f_1 - f_2$ is also bounded on $D$ and satisfies the following integral equation:
\begin{equation*}
\tilde{f}(x, \xi) = \int_0^{\tau_-(x, \xi)} \mu_s(x - s\xi) \exp \Bigl(- M_t (x, \xi; s) \Bigr) \int_{S^{d-1}} p(x - s\xi, \xi, \xi^\prime) \tilde{f} (x - s\xi, \xi^\prime)\,d\sigma_{\xi^\prime}ds
\end{equation*}
for all $(x, \xi) \in D$. Then, we have
\begin{align*}
|\tilde{f}(x, \xi)| \leq& \left( \sup_{(x, \xi) \in D} | \tilde{f} (x, \xi)| \right) \int_0^{\tau_-(x, \xi)} \mu_s(x - s\xi) \exp \Bigl(- M_t (x, \xi; s) \Bigr)ds\\
\leq& \left( \sup_{(x, \xi) \in D} | \tilde{f} (x, \xi)| \right) \int_0^{\tau_-(x, \xi)} \mu_t(x - s\xi) \exp \Bigl(- M_t (x, \xi; s) \Bigr)ds\\
=& - \left( \sup_{(x, \xi) \in D} | \tilde{f} (x, \xi)| \right) \sum_{j = 1}^{l(x, \xi)} \int_{t_{j-1}(x, \xi)}^{t_j(x, \xi)} \dfrac{d}{ds} \exp \Bigl(- M_t (x, \xi; s) \Bigr)\,ds\\
=& \left( \sup_{(x, \xi) \in D} | \tilde{f} (x, \xi)| \right) \left( 1 - \exp \Bigl(- M_t (x, \xi; \tau_-(x, \xi)) \Bigr) \right)\\
\leq& M \left( \sup_{(x, \xi) \in D} | \tilde{f} (x, \xi)| \right)
\end{align*}
for all $(x, \xi) \in D$, where
\begin{equation*}
M := \sup_{(x, \xi) \in D} \left( 1 - \exp \Bigl(- M_t (x, \xi; \tau_-(x, \xi)) \Bigr) \right)
\end{equation*}
and $0 \leq M < 1$ by the boundedness of $\mu_t$. We recall that $l(x, \xi)$ and $\{ t_j(x, \xi) \}_{j=0}^{l(x, \xi)}$ with $t_0(x, \xi) = 0$ are numbers appeared in the statement of the generalized convexity condition. Also, we emphasize that the supremum in this paper is not the essential supremum. Therefore,
\begin{equation*}
\sup_{(x, \xi) \in D} | \tilde{f} (x, \xi)| \leq M \left( \sup_{(x, \xi) \in D} | \tilde{f} (x, \xi)| \right), 
\end{equation*}
and it implies $\sup_{(x, \xi) \in D} | \tilde{f} (x, \xi)| = 0$, that is, $f_1 = f_2$ on $D$.
\end{proof}

We prove existence of a solution by iteration (see \cite{Kawa}). Define a sequence of functions $\{f^{(n)}\}_{n \geq 0}$ on $D$ by
\begin{equation} \label{eq:F0}
f^{(0)}(x, \xi) := \exp \Bigl(- M_t(x, \xi; \tau_- (x, \xi)) \Bigr) f_0(P(x, \xi), \xi),
\end{equation}
and
\begin{align} \label{eq:F1}
f^{(n+1)}(x, \xi) :=& \int_0^{\tau_-(x, \xi)} \mu_s(x - s\xi) \exp \Bigl(- M_t (x, \xi; s) \Bigr) \nonumber\\
&\quad \times \int_{S^{d-1}} p(x - s\xi, \xi, \xi^\prime) f^{(n)}(x - s\xi, \xi^\prime)\,d\sigma_{\xi^\prime}ds.
\end{align}
In fact, the sum $f := \sum_{n = 0}^\infty f^{(n)}$ is a solution to the integral equation (\ref{eq:IE}). To see this, we give the following two propositions.

\begin{Prop} \label{prop:P1}
Suppose that the boundary data $f_0$ is bounded on $\Gamma_-$. Then, each $f^{(n)}$ is also bounded on $D$.
\end{Prop}

\begin{Prop} \label{prop:P2}
Suppose that the boundary data $f_0$ is bounded on $\Gamma_-$. Then, the sum $\sum_{n = 0}^\infty f^{(n)}(x, \xi)$ is absolutely and uniformly convergent on $D$.
\end{Prop}

\begin{proof}[Proof of Proposition \ref{prop:P1}]
We use induction with respect to $n$. For $n = 0$, we have
\begin{equation*}
|f^{(0)}(x, \xi)| \leq \exp \Bigl(- M_t(x, \xi; \tau_- (x, \xi)) \Bigr) |f_0(x-\tau_-(x, \xi)\xi, \xi)| \leq \sup_{(x, \xi) \in \Gamma_-} | f_0 (x, \xi) |
\end{equation*}
for all $(x, \xi) \in \Omega \times S^{d-1}$. This estimate implies that $f^{(0)}$ is bounded on $D$.

In case that $f^{(n)}$ is bounded on $D$ for some $n \in \mathbb{N}$, we have
\begin{align}
| f^{(n+1)} (x, \xi) | \leq& \int_0^{\tau_-(x, \xi)} \mu_s(x - s\xi) \exp \Bigl(- M_t (x, \xi; s) \Bigr) \nonumber \\
&\quad \times \int_{S^{d-1}} p(x - s\xi, \xi, \xi^\prime) | f^{(n)}(x - s\xi, \xi^\prime) | \,d\sigma_{\xi^\prime}ds \nonumber \\
\leq& \left( \sup_{(x, \xi) \in D} | f^{(n)} (x, \xi) | \right)\int_0^{\tau_-(x, \xi)} \mu_s(x - s\xi) \exp \Bigl(- M_t (x, \xi; s) \Bigr)\,ds \nonumber \\
\leq& M \left( \sup_{(x, \xi) \in D} | f^{(n)} (x, \xi) | \right) \label{ineq:E}
\end{align}
for all $(x, \xi) \in D$. This inequality implies that $f^{(n+1)}$ is defined and bounded on $D$. This completes the proof.
\end{proof}

\begin{proof}[Proof of Proposition \ref{prop:P2}]
From the inequality (\ref{ineq:E}), we have, for all $n \geq 0$, 
\begin{equation*}
\sup_{(x, \xi) \in D} | f^{(n)} (x, \xi) | \leq M \left( \sup_{(x, \xi) \in D} | f^{(n-1)} (x, \xi) | \right) \leq M^n \left( \sup_{(x, \xi) \in \Gamma_-} | f_0 (x, \xi) | \right).
\end{equation*}
Thus, 
\begin{equation*}
\sum_{n = 0}^\infty |f^{(n)} (x, \xi)| \leq \sum_{n = 0}^\infty M^n \left( \sup_{(x, \xi) \in \Gamma_-} | f_0 (x, \xi) | \right) = \frac{1}{1 - M} \left( \sup_{(x, \xi) \in \Gamma_-} | f_0 (x, \xi) | \right) < \infty,
\end{equation*}
which implies absolute and uniform convergence of the sum $\sum_{n = 0}^\infty f^{(n)}(x, \xi)$ on $D$.
\end{proof}

From Proposition \ref{prop:P1} and Proposition \ref{prop:P2}, the sum $f(x, \xi) = \sum_{n = 0}^\infty f^{(n)}(x, \xi)$ converges absolutely and uniformly on $D$ and satisfies
\begin{align*}
f(x, \xi) =& f^{(0)}(x, \xi) + \sum_{n = 0}^\infty f^{(n+1)}(x, \xi)\\
=& f^{(0)}(x, \xi) + \int_0^{\tau_-(x, \xi)} \mu_s(x - s\xi) \exp \Bigl(- M_t (x, \xi; s) \Bigr)\\
& \quad \times \int_{S^{d-1}} p(x - s\xi, \xi, \xi^\prime) \sum_{n = 0}^\infty f^{(n)}(x - s\xi, \xi^\prime)\,d\sigma_{\xi^\prime}ds\\
=& f^{(0)}(x, \xi) + \int_0^{\tau_-(x, \xi)} \mu_s(x - s\xi) \exp \Bigl(- M_t (x, \xi; s) \Bigr)\\
&\quad \times \int_{S^{d-1}} p(x - s\xi, \xi, \xi^\prime) f(x - s\xi, \xi^\prime)\,d\sigma_{\xi^\prime}ds
\end{align*}
for all $(x, \xi) \in D$, which is the integral equation (\ref{eq:IE}) itself. Thus, the sum is the unique solution to the integral equation (\ref{eq:IE}).

\section{Discontinuity of the solution}
We discuss discontinuity of the solution to the integral equation (\ref{eq:IE}). To this end, we decompose the solution $f$ into two parts as the following: 
\begin{equation*}
f(x, \xi) = F_0(x, \xi) + F_1(x, \xi),
\end{equation*}
where
\begin{equation*}
F_0(x, \xi) := f^{(0)}(x, \xi), \quad F_1(x, \xi) := \sum_{n = 1}^\infty f^{(n)}(x, \xi).
\end{equation*}
We observe discontinuity of $F_0$ and give a proof of continuity of $F_1$.

We prepare the following lemma in advance.

\begin{Lem} \label{lem:LI}
Let $g$ be a bounded function on $\mathbb{R}^d \times S^{d-1}$ such that it is continuous on $\Omega_0 \times S^{d-1}$ and $g(x, \xi) = 0$ for $(x, \xi) \in (\mathbb{R}^d \backslash \Omega_0) \times S^{d-1}$. Let $R$ be the diameter of the domain $\Omega$. Then, for all $s \in \lbrack 0, R \rbrack$, the integral
\begin{equation} \label{form:LI}
\int_0^s g(x - r\xi, \xi)\,dr
\end{equation}
is continuous at all $(x, \xi) \in \mathbb{R}^d \times S^{d-1}$.
\end{Lem}

\begin{proof}
We fix $(\bar{x}, \bar{\xi}) \in \mathbb{R}^d \times S^{d-1}$ and take $s \in \lbrack 0, R \rbrack$. Since $g$ is continuous on $(\mathbb{R}^d \backslash \partial \Omega_0) \times S^{d-1}$, $g(x - r\xi, \xi)$ converges to $g(\bar{x} - r\bar{\xi}, \bar{\xi})$ as $(x, \xi)$ tends to $(\bar{x}, \bar{\xi})$ except for $r = t_j(\bar{x}, \bar{\xi})$, $j = 1, \ldots, l_s(\bar{x}, \bar{\xi})$, where
\begin{equation*}
l_s(\bar{x}, \bar{\xi}) := \max \{ j \in \{ 1, \ldots, l(\bar{x}, \bar{\xi}) \} | t_j(\bar{x}, \bar{\xi}) \leq s \}.
\end{equation*}
By the generalized convexity condition, $l_s(\bar{x}, \bar{\xi})$ is at most finite. Thus, we apply Lebesgue's convergence theorem to conclude that the integral (\ref{form:LI}) is continuous $(\bar{x}, \bar{\xi}) \in \mathbb{R}^d \times S^{d-1}$ for all $s \in \lbrack 0, R \rbrack$.
\end{proof}

\begin{Cor} \label{cor:LI}
Under the same assumption in  Lemma \ref{lem:LI}, the integral
\begin{equation*}
\int_0^{\tau_-(x, \xi)} g(x - r\xi, \xi)\,dr
\end{equation*}
is continuous at all $(x, \xi) \in D$.
\end{Cor}

\begin{proof}
Since  $g(x, \xi) = 0$ for $(x, \xi) \in (\mathbb{R}^d \backslash \Omega_0) \times S^{d-1}$, we have
\begin{equation*}
\int_0^{\tau_-(x, \xi)} g(x - r\xi, \xi)\,dr = \int_0^R g(x - r\xi, \xi)\,dr
\end{equation*}
for $(x, \xi) \in D$. The right hand side is continuous at all $(x, \xi) \in D$ by Lemma \ref{lem:LI}.
\end{proof}

Now we are ready to discuss discontinuity of the solution. First, we observe discontinuity of $F_0$.
\begin{Prop} \label{prop:F0}
\begin{equation*}
\disc(F_0) = \{(x_* + t\xi_*, \xi_*) | (x_*, \xi_*) \in \disc(f_0), 0 \leq t < \tau_+(x_*, \xi_*) \}.
\end{equation*}
\end{Prop}

\begin{proof}
Let us recall the explicit formula of $F_0$ (\ref{eq:F0}): for all $(x, \xi) \in D$,
\begin{equation*}
F_0(x, \xi) = \exp \Bigl( -M_t(x, \xi; \tau_-(x, \xi))\,dr \Bigr) f_0(P(x, \xi), \xi),
\end{equation*}
where $P(x, \xi) = x - \tau_-(x, \xi)\xi$. Since $\tau_-$ is continuous on $D$ (see \cite{Guo}), Corollary \ref{cor:LI} with $g(x, \xi) = \mu_t(x)$ guarantees that $\exp \Bigl( -M_t(x, \xi; \tau_-(x, \xi)) \Bigr)$ is continuous on $(x, \xi) \in D$. Thus, we have  
\begin{equation*}
(x, \xi) \in \disc(F_0) \Leftrightarrow (x - \tau_-(x, \xi)\xi, \xi) \in \disc(f_0).
\end{equation*}
Let $x_* = x - \tau_-(x, \xi)\xi$. Then, we have $x = x_* + \tau_-(x, \xi)\xi$ and $0 \leq \tau_-(x, \xi) < \tau_+(x_*, \xi)$, which completes the proof.
\end{proof}

Second, we prove continuity of $F_1$. For this, it suffices to prove that functions $f^{(n)}$, defined by (\ref{eq:F0})-(\ref{eq:F1}), are bounded continuous on $D$ for all $n \geq 1$ since we already know from Proposition \ref{prop:P2} that the sum $\sum_{n=1}^\infty f^{(n)}(x, \xi)$ converges uniformly on $D$. In what follows, we discuss continuity of each $f^{(n)}$.

We prepare the following lemma.

\begin{Lem} \label{lem:ELI}
Let $h$ be a bounded function on $\mathbb{R}^d \times S^{d-1}$ such that it is continuous on $\Omega_0 \times S^{d-1}$ and $h(x, \xi) = 0$ for $(x, \xi) \in (\mathbb{R}^d \backslash \Omega_0) \times S^{d-1}$. Let $R$ be the diameter of the domain $\Omega$. Then, the integral
\begin{equation*}
\int_0^R \exp \Bigl( -M_t(x, \xi; s) \Bigr) h(x - s\xi, \xi, \xi^\prime)\,ds
\end{equation*}
is continuous at all $(x, \xi) \in \mathbb{R}^d \times S^{d-1}$.
\end{Lem}

\begin{proof}
We can prove this lemma in the same way as Lemma \ref{lem:LI}.
\end{proof}

\begin{Lem} \label{lem:f1}
Under the assumption in Theorem \ref{thm:MR1}, $f^{(1)}$ is bounded continuous on $D$.
\end{Lem}
\begin{proof}
Boundedness of $f^{(1)}$ was already proved in section 2, and we here prove its continuity. From the explicit formula of $f^{(0)}$ (\ref{eq:F0}), we have
\begin{align*}
f^{(1)}(x, \xi) =& \int_0^{\tau_-(x, \xi)} \mu_s(x - s\xi) \exp \Bigl(- M_t (x, \xi; s) \Bigr)\\
&\quad \times \int_{S^{d-1}} p(x - s\xi, \xi, \xi^\prime) f^{(0)}(x - s\xi, \xi^\prime)\,d\sigma_{\xi^\prime}ds\\
=& \int_0^{\tau_-(x, \xi)} \mu_s(x - s\xi) \exp \Bigl(- M_t (x, \xi; s) \Bigr) G(x - s\xi, \xi)\,ds,
\end{align*}
where
\begin{equation} \label{eq:G}
G(x, \xi) := \int_{S^{d-1}} p(x, \xi, \xi^\prime) \exp \Bigl(- M_t (x, \xi^\prime; \tau_- (x, \xi^\prime)) \Bigr) f_0(P(x, \xi^\prime), \xi^\prime) \,d\sigma_{\xi^\prime}.
\end{equation}
We note that $G$ is defined only on $D$, and we have the following lemma, whose proof will be shown later.

\begin{Lem} \label{lem:G}
Under the assumption in Theorem \ref{thm:MR1}, $G$ is bounded continuous on $\Omega_0 \times S^{d-1}$.
\end{Lem}

Admitting Lemma \ref{lem:G}, we continue to prove Lemma \ref{lem:f1}. Let $\widetilde{G}$ be the zero extension of $G$ to $\mathbb{R}^d \times S^{d-1}$;
\begin{equation*}
\widetilde{G}(x, \xi) :=
\begin{cases}
G(x, \xi), &(x, \xi) \in \Omega_0 \times S^{d-1},\\
0, \quad &otherwise.
\end{cases}
\end{equation*}
Then, $f^{(1)}$ can be written as the following:
\begin{equation*}
f^{(1)}(x, \xi) = \int_0^R \mu_s(x - s\xi) \exp \Bigl(- M_t(x, \xi; s) \Bigr) \widetilde{G}(x - s\xi, \xi)\,ds.
\end{equation*}
Then, the conclusion immediately follows from Lemma \ref{lem:ELI} with $h(x, \xi) = \mu_s(x) \widetilde{G}(x, \xi)$.
\end{proof}

\begin{proof}[Proof of Lemma \ref{lem:G}]
It is obvious that $G$ is bounded because the integrand is also bounded, so we only discuss continuity of $G$. At first we fix a point $(\overline{x}, \overline{\xi}) \in \Omega_0 \times S^{d-1}$, and the we prove continuity of $G$ at the point $(\overline{x}, \overline{\xi})$.

First, we consider the case where $f_0$ satisfies the condition 1 in Theorem \ref{thm:MR1}. In this case, we change a domain of integration appeared in $G$ from $S^{d-1}$ to $\partial \Omega$. 

We investigate a relation between the Lebesgue measure $d\sigma_{\xi^\prime}$ on the unit sphere $S^{d-1}$ and that $d\sigma_y$ on the boundary $\partial \Omega$. Let us introduce the atlas $\{ (U_\lambda, \varphi_\lambda) \}_{\lambda \in \Lambda}$ of $\partial \Omega$. Since $\partial \Omega$ is compact, we can choose a finite number of covering $\{ U_i \}_{i = 1}^N$ from the system of local neighborhoods $\{ U_\lambda \}_{\lambda \in \Lambda}$, that is, $\cup_{i = 1}^N U_i = \partial \Omega$. Also, for $x \in \Omega$, we define a map $P_x : S^{d-1} \rightarrow \partial \Omega$ by
\begin{equation*}
P_x(\xi) := P(x, \xi), \quad \xi \in S^{d-1}.
\end{equation*}
In fact, since $\partial \Omega$ is bounded convex, the map $P_x$ is well-defined for all $x \in \Omega$ and the inverse map $P_x^{-1} : \partial \Omega \rightarrow S^{d-1}$ is written by
\begin{equation} \label{eq:proj}
P_x^{-1}(y) = \frac{x - y}{|x - y|}, \quad y \in \partial \Omega.
\end{equation}
Moreover, since $\partial \Omega$ is $C^1$, the map $P_x$ is a diffeomorphism for all $x \in \Omega$. Thus, we introduce the atlas $\{ (P_x^{-1}(U_i), \varphi_i \circ P_x) \}_{i=1}^N$ on $S^{d-1}$. Then, the following relation holds.

\begin{Lem} \label{lem:J}
For all $i = 1, \ldots, N$,
\begin{equation*}
d\sigma_{\xi^\prime} = \dfrac{|n(y) \cdot (x - y)|}{|x - y|^d}\,d\sigma_y
\end{equation*}
via $y = P_x(\xi^\prime)$, $\xi^\prime \in P_x^{-1}(U_i)$.
\end{Lem}

We are ready to change the domain of integration from $S^{d-1}$ to $\partial \Omega$. Let $\{ \rho_i \}_{i = 1}^N$ be the partition of unity on $S^{d-1}$ corresponding to $\{ P_x^{-1} (U_i) \}_{i  = 1}^N$, that is, $0 \leq \rho_i \leq 1$ on $S^{d-1}$ for all $i$, supp$\rho_i$ $\subset P_x^{-1} ( U_i )$ for all $i$, and $\sum_{i=1}^N \rho_i(\xi) = 1$ for all $\xi \in S^{d-1}$. Then, we have
\begin{equation*}
G(x, \xi) = \sum_{i=1}^N \int_{P_x^{-1}(U_i)} p(x, \xi, \xi^\prime) \exp \Bigl(- M_t (x, \xi^\prime; \tau_- (x, \xi^\prime)) \Bigl) f_0(P_x(\xi^\prime), \xi^\prime) \rho_i(\xi^\prime)\,d\sigma_{\xi^\prime}.
\end{equation*}
From Lemma \ref{lem:J}, we have
\begin{align*}
G(x, \xi) =& \sum_{i=1}^N \int_{U_i} p \left(x, \xi, \frac{x - y}{|x - y|} \right) \exp \left(- M_t \left(x, \frac{x - y}{|x - y|}; |x - y|\right) \right)\\
&\quad \times f_0 \left(y, \frac{x - y}{|x - y|} \right) \rho_i \left(\frac{x - y}{|x - y|} \right) \dfrac{ |n(y) \cdot (x - y) |} {|x - y|^d} \,d\sigma_y\\
=& \int_{\partial \Omega} p \left(x, \xi, \frac{x - y}{|x - y|} \right) \exp \left(- M_t \left(x, \frac{x - y}{|x - y|}; |x - y|\right) \right) f_0 \left(y, \frac{x - y}{|x - y|} \right) \dfrac{ |n(y) \cdot (x - y) |} {|x - y|^d} \,d\sigma_y.
\end{align*}
    
We prove continuity of $G$ at $(\overline{x}, \overline{\xi}) \in \Omega_0 \times S^{d-1}$. Let $\epsilon := d(\overline{x}, \partial \Omega_0)$. Then, the integrand
\begin{equation*}
p \left(x, \xi, \frac{x - y}{|x - y|} \right) \exp \left(- M_t \left(x, \frac{x - y}{|x - y|}; |x - y|\right) \right) f_0 \left(y, \frac{x - y}{|x - y|} \right) \dfrac{\left| n(y) \cdot (x - y) \right|}{|x - y|^d}
\end{equation*}
is bounded continuous on $B_{\epsilon / 2}(\overline{x}) \times S^{d-1}$ for almost all $y \in \partial \Omega$, where $B_{\epsilon / 2}(\overline{x})$ is the open ball in $\mathbb{R}^d$ centered at $\overline{x}$ with radius $\epsilon / 2$. Thus, we apply the dominated convergence theorem to conclude that $G$ is continuous at $(\overline{x}, \overline{\xi}) \in \Omega_0 \times S^{d-1}$. 

Second, we consider the case where the boundary data $f_0$ satisfies the condition 2 in Theorem \ref{thm:MR1}, which is done by Anikonov et al. \cite{1993Anik}. Because of convexity of the domain $\Omega$ and smoothness of the boundary $\partial \Omega$, $\tau_-$ is continuous on $D$ (see \cite{Guo}). Thus, for almost all $\xi^\prime \in S^{d-1}$, the integrand 
\begin{equation*}
p(x, \xi, \xi^\prime) \exp \Bigl(- M_t (x, \xi^\prime; \tau_- (x, \xi^\prime)) \Bigr) f_0(P(x, \xi^\prime), \xi^\prime)
\end{equation*}
is continuous at $(\overline{x}, \overline{\xi}) \in \Omega_0 \times S^{d-1}$. Furthermore, the integrand is bounded by 
\begin{equation*}
\left( \sup_{(x, \xi, \xi^\prime) \in \Omega_0 \times S^{d-1} \times S^{d-1}} p(x, \xi, \xi^\prime) \right) \left( \sup_{(x, \xi) \in \Gamma_-} |f_0(x, \xi)| \right),
\end{equation*}
which is obviously integrable with respect to $\xi^\prime$. Therefore, we apply the dominated convergence theorem to conclude that $G$ is bounded continuous on $\Omega_0 \times S^{d-1}$.

We finish the proof of Lemma \ref{lem:G}.
\end{proof}

\begin{proof}[Proof of Lemma \ref{lem:J}]
A proof depends on its dimension $d$. 

We consider the two dimensional case $d=2$. In this case, $U_i$ is parametrized by an interval $(a_i, b_i)$ via $\varphi_i(U_i) = (a_i, b_i)$. From this parametrization, the Lebesgue measure $d\sigma_{\xi^\prime}$ on $S^1$ is given by
\begin{equation*}
d\sigma_{\xi^\prime} = \sqrt{ \left( \dfrac{\partial \xi^\prime_1}{\partial t} \right)^2 + \left( \dfrac{\partial \xi^\prime_2}{\partial t} \right)^2 }\,dt, \quad t \in (a_i, b_i).
\end{equation*}
By the chain rule, we have
\begin{equation*}
\dfrac{\partial \xi^\prime_1}{\partial t} = \dfrac{\partial \xi^\prime_1}{\partial y_1} \dfrac{dy_1}{dt} + \dfrac{\partial \xi^\prime_1}{\partial y_2} \dfrac{dy_2}{dt}, \quad \dfrac{\partial \xi^\prime_2}{\partial t} = \dfrac{\partial \xi^\prime_2}{\partial y_1} \dfrac{dy_1}{dt} + \dfrac{\partial \xi^\prime_2}{\partial y_2} \dfrac{dy_2}{dt}.
\end{equation*}
Then, we have 
\begin{align*}
\left( \dfrac{\partial \xi^\prime_1}{\partial t} \right)^2 + \left( \dfrac{\partial \xi^\prime_2}{\partial t} \right)^2 =& \left\{ \left( \dfrac{\partial \xi^\prime_1}{\partial y_1} \right)^2 + \left( \dfrac{\partial \xi^\prime_2}{\partial y_1} \right)^2 \right\} \left( \dfrac{dy_1}{dt} \right)^2 + \left\{ \left( \dfrac{\partial \xi^\prime_1}{\partial y_2} \right)^2 + \left( \dfrac{\partial \xi^\prime_2}{\partial y_2} \right)^2 \right\} \left( \dfrac{dy_2}{dt} \right)^2\\
& + 2 \left\{ \left( \dfrac{\partial \xi^\prime_1}{\partial y_1} \right) \left( \dfrac{\partial \xi^\prime_1}{\partial y_2} \right) + \left( \dfrac{\partial \xi^\prime_2}{\partial y_1} \right) \left( \dfrac{\partial \xi^\prime_2}{\partial y_2} \right) \right\} \left( \dfrac{dy_1}{dt} \right) \left( \dfrac{dy_2}{dt} \right).
\end{align*}
Since 
\begin{equation*}
\dfrac{\partial \xi^\prime_i}{\partial y_j} = - \dfrac{\delta_{i, j}}{|x - y|} + \dfrac{(x_i - y_i)(x_j - y_j)}{|x - y|^3},
\end{equation*}
where $\delta_{i, j}$ is the Kronecker's delta, for $i, j = 1, 2$, we have
\begin{align*}
\left( \dfrac{\partial \xi^\prime_1}{\partial y_1} \right)^2 + \left( \dfrac{\partial \xi^\prime_2}{\partial y_1} \right)^2 =& \left( -\dfrac{1}{|x - y|} + \dfrac{(x_1 - y_1)^2}{|x - y|^3} \right)^2 + \left( \dfrac{(x_1 - y_1)(x_2 - y_2)}{|x - y|^3} \right)^2\\
=& \dfrac{(x_2 - y_2)^4}{|x - y|^6} + \dfrac{(x_1 - y_1)^2 (x_2 - y_2)^2}{|x - y|^6} = \dfrac{(x_2 - y_2)^2}{|x - y|^4},
\end{align*}
\begin{align*}
\left( \dfrac{\partial \xi^\prime_1}{\partial y_2} \right)^2 + \left( \dfrac{\partial \xi^\prime_2}{\partial y_2} \right)^2 =& \left( \dfrac{(x_2 - y_2)(x_1 - y_1)}{|x - y|^3} \right)^2 + \left( -\dfrac{1}{|x - y|} + \dfrac{(x_2 - y_2)^2}{|x - y|^3} \right)^2\\
=& \dfrac{(x_1 - y_1)^2 (x_2 - y_2)^2}{|x - y|^6}  + \dfrac{(x_1 - y_1)^4}{|x - y|^6} = \dfrac{(x_1 - y_1)^2}{|x - y|^4},
\end{align*}
and
\begin{align*}
2 \left\{ \left( \dfrac{\partial \xi^\prime_1}{\partial y_1} \right) \left( \dfrac{\partial \xi^\prime_1}{\partial y_2} \right) + \left( \dfrac{\partial \xi^\prime_2}{\partial y_1} \right) \left( \dfrac{\partial \xi^\prime_2}{\partial y_2} \right) \right\} =& 2 \left( -\dfrac{1}{|x - y|} + \dfrac{(x_1 - y_1)^2}{|x - y|^3} \right) \left( \dfrac{(x_1 - y_1)(x_2 - y_2)}{|x - y|^3} \right)\\
&+ 2 \left( \dfrac{(x_2 - y_2)(x_1 - y_1)}{|x - y|^3} \right) \left( -\dfrac{1}{|x - y|} + \dfrac{(x_2 - y_2)^2}{|x - y|^3} \right)\\
=& -2 \dfrac{(x_1 - y_1)(x_2 - y_2)^3}{|x - y|^6} -2 \dfrac{(x_2 - y_2)(x_1 - y_1)^3}{|x - y|^6}\\
=& -2 \dfrac{(x_1 - y_1)(x_2 - y_2)}{|x - y|^4}.  
\end{align*}
Thus, we have
\begin{align*}
d\sigma_{\xi^\prime} =& \dfrac{1}{|x - y|^2} \left\{ (x_2 - y_2)^2 \left(\dfrac{dy_1}{dt} \right)^2 \right.\\
&\left. -2 (x_1 - y_1)(x_2 - y_2) \left(\dfrac{dy_1}{dt} \right) \left(\dfrac{dy_2}{dt} \right) + (x_1 - y_1)^2 \left(\dfrac{dy_2}{dt} \right)^2 \right\}^{1/2} \,dt\\
=& \dfrac{1}{|x - y|^2} \left| (x_2 - y_2) \left(\dfrac{dy_1}{dt} \right) -  (x_1 - y_1) \left(\dfrac{dy_2}{dt} \right) \right|\,dt\\
=& \dfrac{\left|(x - y) \cdot n \bigl( y(t) \bigr) \right|}{|x - y|^2} \sqrt{\left(\dfrac{dy_1}{dt} \right)^2 + \left(\dfrac{dy_2}{dt} \right)^2}\,dt = \dfrac{\left|(x - y) \cdot n(y) \right|}{|x - y|^2}\,d\sigma_y,
\end{align*}
where $d\sigma_y$ is the Lebesgue measure on $\partial \Omega$. We note that, since 
\begin{equation*}
\left( \dfrac{d y_1}{dt}(t), \dfrac{d y_2}{dt}(t) \right)
\end{equation*}
is a tangent vector at $y(t) \in \partial \Omega$, the vector
\begin{equation*}
\dfrac{1}{\sqrt{\left(\dfrac{dy_1}{dt} \right)^2 + \left(\dfrac{dy_2}{dt} \right)^2}} \left( -\dfrac{dy_2}{dt}(t), \dfrac{dy_1}{dt}(t) \right)
\end{equation*}
is a unit normal vector at $y(t) \in \partial \Omega$. 

We consider the three dimensional case $d = 3$. Let $(p_i, q_i) \in \mathbb{R}^2$ be the parametrization of $(U_i, \varphi_i)$. Then, $(P_x^{-1}(U_i), \varphi_i \circ P_x)$ has the same parametrization and
\begin{equation*}
d\sigma_{\xi^\prime} = \left| \frac{\partial \xi^\prime}{\partial p_i} \times \frac{\partial \xi^\prime}{\partial q_i} \right|\,dp_i dq_i, \quad (p_i, q_i) \in \varphi_i (U_i).
\end{equation*}
In the same way, we have
\begin{equation*}
d\sigma_y = \left| \frac{\partial y}{\partial p_i} \times \frac{\partial y}{\partial q_i} \right|\,dp_i dq_i, \quad (p_i, q_i) \in \varphi_i (U_i).
\end{equation*}

By the chain rule of differentiation, we have
\begin{equation*}
\frac{\partial \xi^\prime}{\partial p_i} = \sum_{j=1}^3 \dfrac{\partial y_j}{\partial p_i} \dfrac{\partial \xi^\prime}{\partial y_j}
\end{equation*}
and
\begin{align*}
\frac{\partial \xi^\prime}{\partial p_i} \times \frac{\partial \xi^\prime}{\partial q_i} =& \left( \sum_{j=1}^3 \dfrac{\partial y_j}{\partial p_i} \dfrac{\partial \xi^\prime}{\partial y_j} \right)
\times \left( \sum_{k=1}^3 \dfrac{\partial y_k}{\partial q_i} \dfrac{\partial \xi^\prime}{\partial y_k} \right)\\
=& \sum_{j, k=1}^3 \dfrac{\partial y_j}{\partial p_i} \dfrac{\partial y_k}{\partial q_i} \left( \dfrac{\partial \xi^\prime}{\partial y_j} \times \dfrac{\partial \xi^\prime}{\partial y_k} \right)\\
=& \left( \dfrac{\partial y_1}{\partial p_i} \dfrac{\partial y_2}{\partial q_i} - \dfrac{\partial y_2}{\partial p_i} \dfrac{\partial y_1}{\partial q_i} \right) \left( \dfrac{\partial \xi^\prime}{\partial y_1} \times \dfrac{\partial \xi^\prime}{\partial y_2} \right)\\
&+ \left( \dfrac{\partial y_1}{\partial p_i} \dfrac{\partial y_3}{\partial q_i} - \dfrac{\partial y_3}{\partial p_i} \dfrac{\partial y_1}{\partial q_i} \right) \left( \dfrac{\partial \xi^\prime}{\partial y_1} \times \dfrac{\partial \xi^\prime}{\partial y_3} \right)\\
&+ \left( \dfrac{\partial y_2}{\partial p_i} \dfrac{\partial y_3}{\partial q_i} - \dfrac{\partial y_3}{\partial p_i} \dfrac{\partial y_2}{\partial q_i} \right) \left( \dfrac{\partial \xi^\prime}{\partial y_2} \times \dfrac{\partial \xi^\prime}{\partial y_3} \right).
\end{align*}

From the explicit formula of $P_x^{-1}$ (\ref{eq:proj}), we have 
\begin{equation*}
\frac{\partial \xi_i^\prime}{\partial y_j} = -\frac{\delta_{i, j}}{|x - y|} + \frac{(x_i - y_i)(x_j - y_j)}{|x - y|^3},
\end{equation*}
where $\delta_{i, j}$ is the Kronecker's delta for $i, j = 1, 2, 3$. Thus, we have
\begin{align*}
\dfrac{\partial \xi^\prime}{\partial y_1} \times \dfrac{\partial \xi^\prime}{\partial y_2} &= \dfrac{x_3 - y_3}{|x - y|^4} (x - y), \\
\dfrac{\partial \xi^\prime}{\partial y_1} \times \dfrac{\partial \xi^\prime}{\partial y_3} &= -\dfrac{x_2 - y_2}{|x - y|^4} (x - y), \\
\dfrac{\partial \xi^\prime}{\partial y_2} \times \dfrac{\partial \xi^\prime}{\partial y_3} &= \dfrac{x_1 - y_1}{|x - y|^4} (x - y).
\end{align*}
From these formulae, we have
\begin{align*}
\frac{\partial \xi^\prime}{\partial p_i} \times \frac{\partial \xi^\prime}{\partial q_i} =& \dfrac{1}{|x - y|^4} \left\{ \left( \dfrac{\partial y_1}{\partial p_i} \dfrac{\partial y_2}{\partial q_i} - \dfrac{\partial y_2}{\partial p_i} \dfrac{\partial y_1}{\partial q_i} \right) (x_3 - y_3) \right.\\
&- \left( \dfrac{\partial y_1}{\partial p_i} \dfrac{\partial y_3}{\partial q_i} - \dfrac{\partial y_3}{\partial p_i} \dfrac{\partial y_1}{\partial q_i} \right) (x_2 - y_2) \\
&\left. + \left( \dfrac{\partial y_2}{\partial p_i} \dfrac{\partial y_3}{\partial q_i} - \dfrac{\partial y_3}{\partial p_i} \dfrac{\partial y_2}{\partial q_i} \right) (x_1 - y_1) \right\} (x - y)\\
=& \dfrac{1}{|x - y|^4} \left\{ \left( \frac{\partial y}{\partial p_i} \times \frac{\partial y}{\partial q_i} \right) \cdot (x - y) \right\} (x - y)\\
=& \dfrac{ \left| \frac{\partial y}{\partial p_i} \times \frac{\partial y}{\partial q_i} \right| } {|x - y|^4} \left\{ n(y) \cdot (x - y) \right\} (x - y)
\end{align*}
and
\begin{equation*}
\left| \frac{\partial \xi^\prime}{\partial p_i} \times \frac{\partial \xi^\prime}{\partial q_i} \right| = \dfrac{ |n(y) \cdot (x - y) |} {|x - y|^3} \left| \frac{\partial y}{\partial p_i} \times \frac{\partial y}{\partial q_i} \right|. 
\end{equation*}
Here, we note that the vector
\begin{equation*}
\dfrac{1}{\left| \frac{\partial y}{\partial p_i} \times \frac{\partial y}{\partial q_i} \right|} \frac{\partial y}{\partial p_i} \times \frac{\partial y}{\partial q_i} 
\end{equation*}
is a unit normal vector at $y(p, q) \in \partial \Omega$. Thus, we have
\begin{equation*}
d\sigma_{\xi^\prime} = \left| \frac{\partial \xi^\prime}{\partial p_i} \times \frac{\partial \xi^\prime}{\partial q_i} \right|\,dp_i dq_i = \dfrac{ |n(y) \cdot (x - y) |} {|x - y|^3} \left| \frac{\partial y}{\partial p_i} \times \frac{\partial y}{\partial q_i} \right|\,dp_i dq_i = \dfrac{ |n(y) \cdot (x - y) |} {|x - y|^3} d\sigma_y. \qedhere
\end{equation*}
\end{proof}

\begin{Lem} \label{lem:fn}
Suppose that the function $f^{(n)}$, defined by the recursion formula $(\ref{eq:F0})$-$(\ref{eq:F1})$, is bounded continuous on $D$ for some $n \in \mathbb{N}$. Then, the successive function $f^{(n+1)}$ is also bounded continuous on $D$. 
\end{Lem}

\begin{proof}
In the same way as the proof of Lemma \ref{lem:f1}, let $\widetilde{f}^{(n)}$ be the zero extention of $f^{(n)}$ to $\mathbb{R}^d \times S^{d-1}$. Then, we have
\begin{equation*}
f^{(n+1)}(x, \xi) = \int_0^R \mu_s(x - s\xi) \exp \Bigl(- M_t(x, \xi; s) \Bigr) \int_{S^{d-1}} p(x - s\xi, \xi, \xi^\prime) \widetilde{f}^{(n)}(x - s\xi, \xi^\prime)\,d\sigma_{\xi^\prime}ds
\end{equation*}
for all $(x, \xi) \in D$. The conclusion follows from Lemma \ref{lem:ELI} with $h(x, \xi) = \mu_s(x) \int_{S^{d-1}} p(x, \xi, \xi^\prime) \widetilde{f}^{(n)}(x, \xi^\prime)$.
\end{proof}
 
Thus, we succeed to separate the solution into two parts, the discontinuous part $F_0$ and the continuity part $F_1$.

\section{Equivalence between the stationary transport equation and the derived integral equation} \label{subsec:SAI}
We state the equivalence between the boundary value problem (\ref{eq:STE})-(\ref{eq:BC}) and the integral equation (\ref{eq:IE}). It is obvious that the solution $f(\cdot, \xi)$ to the integral equation is continuous on the line segment $\{ x + t\xi | t \in \mathbb{R}\} \cap (\Omega \cup \Gamma_{-, \xi})$ for all $(x, \xi) \in D$. Thus, we show that  the solution to the integral equation satisfies the other conditions of a solution to the boundary value problem. To this end, we show the following lemma.

\begin{Lem} \label{lem:I}
Let $f$ be the solution to the integral equation $(\ref{eq:IE})$. Under the assumption in Theorem \ref{thm:MR1}, the integral
\begin{equation*}
\int_{S^{d-1}} p(x, \xi, \xi^\prime) f(x, \xi^\prime)\,d\sigma_{\xi^\prime}
\end{equation*}
is bounded continuous for $(x, \xi) \in \Omega_0 \times S^1$.
\end{Lem}

\begin{proof}
From the decomposition in section 3, we have
\begin{equation*}
\int_{S^{d-1}} p(x, \xi, \xi^\prime) f(x, \xi^\prime)\,d\sigma_{\xi^\prime} = \int_{S^{d-1}} p(x, \xi, \xi^\prime) F_0(x, \xi^\prime)\,d\sigma_{\xi^\prime} + \int_{S^{d-1}} p(x, \xi, \xi^\prime) F_1(x, \xi^\prime)\,d\sigma_{\xi^\prime}.
\end{equation*}
From Lemma \ref{lem:G}, the function $G(x, \xi) = \int_{S^{d-1}} p(x, \xi, \xi^\prime) F_0(x, \xi^\prime)\,d\sigma_{\xi^\prime}$ is bounded continuous on $\Omega_0 \times S^{d-1}$. On the other hand, $\int_{S^{d-1}} p(x, \xi, \xi^\prime) F_1(x, \xi^\prime)\,d\sigma_{\xi^\prime}$ is also bounded continuous on $\Omega_0 \times S^{d-1}$ since $F_1$ is bounded continuous on $D$ as we see in section 3. Thus, \\
$\int_{S^{d-1}} p(x, \xi, \xi^\prime) f(x, \xi^\prime)\,d\sigma_{\xi^\prime}$ is also bounded continuous on $\Omega_0 \times S^{d-1}$.
\end{proof}

Now we are ready to check the properties of a solution to the boundary value problem (\ref{eq:STE})-(\ref{eq:BC}). From the integral equation (\ref{eq:IE}), noting that $\tau_-(x + t\xi, \xi) = \tau_-(x, \xi) + t$ and $P(x + t\xi, \xi) = P(x, \xi)$, we have 
\begin{align*}
f(x + t\xi, \xi) =& \exp \Bigl(- M_t \bigl(x + t\xi, \xi; \tau_-(x + t\xi, \xi) \bigr) \Bigr) f_0(P(x + t\xi, \xi), \xi)\\
&+ \int_0^{\tau_-(x + t\xi, \xi)} \mu_s(x + t\xi - s\xi) \exp \Bigl( - M_t (x + t\xi, \xi; s) \Bigr)\\
&\quad \times \int_{S^{d-1}} p(x + t\xi - s\xi, \xi, \xi^\prime) f(x + t\xi - s\xi, \xi^\prime)\,d\sigma_{\xi^\prime}ds\\
=& \exp \bigl(M_t \left(x, \xi; -t \right) \bigr) \exp \Bigl(- M_t \bigl(x, \xi; \tau_-(x, \xi) \bigr) \Bigr) f_0(P(x, \xi), \xi)\\
&+ \exp \bigl(M_t \left(x, \xi; -t \right) \bigr) \int_{-t}^{\tau_-(x, \xi)} \mu_s(x -  s\xi) \exp \Bigl( - M_t (x, \xi; s) \Bigr)\\
&\quad \times \int_{S^{d-1}} p(x - s\xi, \xi, \xi^\prime) f(x - s\xi, \xi^\prime)\,d\sigma_{\xi^\prime}ds.
\end{align*}
From Lemma \ref{lem:I}, we can take the directional derivative of $f$:
\begin{align*}
\xi \cdot \nabla_x f(x, \xi) =& \left. \dfrac{d}{dt} f(x + t\xi, \xi) \right|_{t = 0}\\ 
=& -\mu_t(x) \exp \Bigl(- M_t \bigl(x, \xi; \tau_-(x, \xi) \bigr) \Bigr) f_0(P(x, \xi), \xi)\\
&-\mu_t(x) \int_0^{\tau_-(x, \xi)} \mu_s(x -  s\xi) \exp \Bigl( - M_t (x, \xi; s) \Bigr)\\
&\quad \times \int_{S^{d-1}} p(x - s\xi, \xi, \xi^\prime) f(x - s\xi, \xi^\prime)\,d\sigma_{\xi^\prime}ds\\
&+ \mu_s(x) \int_{S^{d-1}} p(x, \xi, \xi^\prime) f(x, \xi^\prime)\,d\sigma_{\xi^\prime}ds\\ 
=& -\mu_t(x) f(x, \xi) + \mu_s(x) \int_{S^{d-1}} p(x, \xi, \xi^\prime) f(x, \xi^\prime)\,d\sigma_{\xi^\prime}ds
\end{align*} 
for all $(x, \xi) \in \Omega_0 \times S^{d-1}$. Thus, the solution $f$ has the directional derivative $\xi \cdot \nabla_x f(x, \xi)$ and satisfies the stationary transport equation (\ref{eq:STE}) for all $(x, \xi) \in \Omega_0 \times S^{d-1}$. Moreover, for $(x, \xi) \in \Omega_0 \times S^{d-1}$, the directional derivative $\xi \cdot \nabla_x f(\cdot, \xi)$ is continuous on the line segments $\{ x + t\xi | t \in ( t_{j-1}(x, \xi), t_j(x, \xi) )\}$, $j = 1, \ldots, l(x, \xi)$ with $t_0(x, \xi)$ = 0. 

Finally, for all $(x, \xi) \in \Gamma_-$, 
\begin{equation*}
f^{(n)}(x, \xi) = 
\begin{cases}
f_0(x, \xi), &n = 0, \\
0, &n \geq 1.
\end{cases}
\end{equation*}
Thus,
\begin{equation*}
f(x, \xi) = \sum_{n=0}^\infty f^{(n)}(x, \xi) = f_0(x, \xi)
\end{equation*}
for all $(x, \xi) \in \Gamma_-$, which implies that the solution $f$ satisfies the incoming boundary condition (\ref{eq:BC}).

Therefore, the boundary value problem of the stationary transport equation (\ref{eq:STE})-(\ref{eq:BC}) and the integral equation (\ref{eq:IE}) are equivalent in this framework.

\section{Discontinuity of the extended solution} \label{subsec:DES}
We extend the definition domain of the solution $f$ up to the outgoing boundary $\Gamma_+$ and discuss the boundary-induced discontinuity of the extended solution.

First, we extend the definition domain of the solution $f$ up to the outgoing boundary $\Gamma_+$. The idea of extension originates from Cessenat \cite{Ces}, who defined it in a weak sense.

\begin{Lem}
For all $(x, \xi) \in \Gamma_+$, the limit
\begin{equation*}
\overline{f}(x, \xi) := \lim_{t \downarrow 0} f(x - t\xi, \xi)
\end{equation*}
exists.
\end{Lem}

\begin{proof}
By the fundamental theorem of calculus, we have
\begin{equation} \label{eq:FTC}
f(x, \xi) = f_0(P(x, \xi), \xi) - \sum_{j = 1}^{l(x, \xi)} \int_{t_{j-1}(x, \xi)}^{t_j(x, \xi)} \xi \cdot \nabla_x f(x - s\xi, \xi)\,ds.
\end{equation}
From the boundedness of the directional derivative $\xi \cdot \nabla_x f(x, \xi)$ the integral in the right hand side of equation (\ref{eq:FTC}) is even defined for all $(x, \xi) \in \Gamma_+$. Thus, we define $f|_{\Gamma_+}$ by equation (\ref{eq:FTC}). 

From equation (\ref{eq:FTC}), for $(x, \xi) \in \Gamma_+$ and for sufficiently small $t > 0$,
\begin{align*}
f(x - t\xi, \xi) =& f_0(P(x - t\xi, \xi), \xi) - \sum_{j = 1}^{l(x - t\xi, \xi)} \int_{t_{j-1}(x - t\xi, \xi)}^{t_j(x - t\xi, \xi)} \xi \cdot \nabla_x f(x - t\xi - s\xi, \xi)\,ds\\
=& f_0(P(x, \xi), \xi) - \sum_{j = 1}^{l(x, \xi)} \int_{t_{j-1}(x, \xi)}^{t_j(x, \xi)} \xi \cdot \nabla_x f(x - s\xi, \xi)\,ds + \int_0^t \xi \cdot \nabla_x f(x - s\xi, \xi)\,ds.
\end{align*}
Here, we used the relation $\tau_-(x - t\xi, \xi) = \tau_-(x, \xi) - t$. Thus, we have
\begin{align*}
|f(x, \xi) - f(x - t\xi, \xi)| =& \left|\int_0^t \xi \cdot \nabla_x f(x - s\xi, \xi)\,ds \right|\\
\leq& \left( \sup_{(x, \xi) \in \Omega \times S^{d-1}} |\xi \cdot \nabla_x f(x, \xi)| \right)t.
\end{align*}
By this estimate, we conclude the proof.
\end{proof}
From this observation, we can extend the definition domain of the solution $f$ up to $\Gamma_+$. 

Second, we discuss the boundary-induced discontinuity of the extended solution. Let $\overline{D} := D \cup \Gamma_+$ and let $\overline{f}$ be the extended solution to the boundary value problem (\ref{eq:STE})-(\ref{eq:BC}) up to $\Gamma_+$. Since the recursion formulae (\ref{eq:F0})-(\ref{eq:F1}) make sense even for $(x, \xi) \in \Gamma_+$, we extend the sequence of functions $\{ f^{(n)} \}_{n \geq 0}$ up to $\Gamma_+$, which is denoted by $\{ \overline{f}^{(n)} \}_{n \geq 0}$. In the same way as in section 3, we can show that the following propositions. In what follows, we use the following notations:
\begin{align*}
\overline{F_0}(x, \xi) &:= \overline{f}^{(0)}(x, \xi), \\
\overline{F_1}(x, \xi) &:= \sum_{n = 1}^\infty \overline{f}^{(n)}(x, \xi).
\end{align*}

\begin{Prop} \label{prop:EF0}
\begin{equation*}
\disc(\overline{F_0}) = \{(x_* + t\xi_*, \xi_*) | (x_*, \xi_*) \in \disc(f_0), 0 \leq t \leq \tau_+(x_*, \xi_*) \}.
\end{equation*}
\end{Prop}

\begin{Prop} \label{prop:EF1}
Under the assumption in Theorem \ref{thm:MR1}, $\overline{F_1}$ is bounded countinuous on $\overline{D}$.
\end{Prop}

Proposition \ref{prop:EF0} and Proposition \ref{prop:EF1} imply Theorem \ref{thm:MR2}.

\section{Decay of jump discontinuity} \label{subsec:DJD}
We discuss the decay of a jump of the boundary-induced discontinuity in the situation introduced in section 1. Let us recall the situation of Theorem \ref{thm:MR3}. For the case $d = 3$, let $\gamma$ be a simple closed curve in $\partial \Omega$, or for the case $d = 2$, let $\gamma$ be two points in $\partial \Omega$, and we note that $\gamma$ decomposes $\partial \Omega$ into two connected components $A$ and $B$ such that $\partial \Omega = A \cup B \cup \gamma$ and $A \cap B = A \cap \gamma = B \cap \gamma = \emptyset$. Suppose
\begin{equation*}
f_0(x, \xi) =
\begin{cases}
I, &(x, \xi) \in ( (A \cup \gamma) \times S^{d-1} ) \cap \Gamma_-,\\
0, &(x, \xi) \in (B \times S^{d-1} ) \cap \Gamma_-,
\end{cases}
\end{equation*} 
where $I$ is a constant. $f_0$ satisfies the condition 2 of Theorem \ref{thm:MR1} since $f_0(x, \cdot)$ is continuous on $\Gamma_{-, x}$ for all $x \in \partial \Omega$. Here, we remark that $\disc(f_0) = \{ (x_*, \xi_*) | x_* \in \gamma, \xi_* \in \Gamma_{-, x_*} \}$. 

For  $(\overline{x}, \overline{\xi}) \in \disc(f)$, we define a jump $[f]$ on $\disc(f)$ by
\begin{equation*}
[f](\overline{x}, \overline{\xi}) := \lim_{\substack{x \rightarrow \overline{x} \\ P(x, \overline{\xi}) \in (A \cup \gamma)}} f(x, \overline{\xi}) - \lim_{\substack{x \rightarrow \overline{x} \\ P(x, \overline{\xi}) \in B}} f(x, \overline{\xi}), 
\end{equation*}
where $P(x, \xi) := x - \tau_-(x, \xi)\xi$. We note that, in our setting, $[f_0](x, \xi) = I$ for all $(x, \xi) \in \disc(f_0) = (\gamma \times S^{d-1}) \cap \Gamma_-$. In this situation, we have the following lemma.

\begin{Lem} \label{lem:DD0}
For $(x^*, \xi^*) \in \disc(\overline{f})$,
\begin{equation*}
[\overline{F_0}](x^*, \xi^*) = I \exp \Bigl( -M_t(x^*, \xi^*; \tau_-(x^*, \xi^*) )\Bigr).
\end{equation*}
\end{Lem}

\begin{proof}
Let $(x^*, \xi^*) \in \disc(\overline{f})$. For $\xi^* \in S^{d-1}$, we introduce the following two domains:
\begin{align*}
\overline{\Omega_{A, \xi^*}} :=& \{ x \in \overline{\Omega} | P(x, \xi^*) \in (A \cup \gamma) \cap \Gamma_{-, \xi^*} \},\\
\overline{\Omega_{B, \xi^*}} :=& \{ x \in \overline{\Omega} | P(x, \xi^*) \in B \cap \Gamma_{-, \xi^*} \}.
\end{align*}
We show that these two sets are not empty. Let $(x_*, \xi_*) := (x^* - \tau_-(x^*, \xi^*)\xi^*, \xi^*)$. Then, the proof of Proposition \ref{prop:F0} shows $(x_*, \xi_*) \in \disc(f_0)$, which means $x_* \in \gamma$. Let $\delta > 0$. Then, $B_\delta(x_*)$ intersects both $A$ and $B$. Take $x_A \in B_\delta(x_*) \cap A$ and $x_B \in B_\delta(x_*) \cap B$. Since $n(x_*) \cdot \xi_* < 0$ and the outer unit normal vector $n(x)$ is continuous on $\partial \Omega$, we note that $n(x_A) \cdot \xi_* < 0$ and $n(x_B) \cdot \xi_* < 0$ for sufficiently small $\delta$. Thus, $x_A + t\xi_* \in \overline{\Omega_{A, \xi^*}}$ and $x_B + t\xi_* \in \overline{\Omega_{B, \xi^*}}$ for small $t > 0$, which implies that neither $\overline{\Omega_{A, \xi^*}}$ nor $\overline{\Omega_{B, \xi^*}}$ is empty.

We are ready to prove Lemma \ref{lem:DD0}. For $x \in \overline{\Omega_{A, \xi^*}}$, 
\begin{equation*}
\overline{F_0}(x, \xi^*) = I \exp \Bigl( -M_t(x, \xi^*; \tau_-(x, \xi^*)) \Bigr).
\end{equation*}
On the other hand, for $x \in \overline{\Omega_{B, \xi^*}}$, 
\begin{equation*}
\overline{F_0}(x, \xi^*) = 0.
\end{equation*}
Thus, we have
\begin{equation*}
[\overline{F_0}](x^*, \xi^*) = I \exp \Bigl( -M_t(x, \xi^*; \tau_-(x, \xi^*)) \Bigr). \qedhere
\end{equation*}
\end{proof}

\begin{Lem} \label{lem:DD1}
For $(x^*, \xi^*) \in \disc(\overline{f})$,
\begin{equation*}
[\overline{F_1}](x^*, \xi^*) = 0.
\end{equation*}
\end{Lem}

\begin{proof}
It immediately follows from Proposition \ref{prop:EF1}.
\end{proof}

From Lemma \ref{lem:DD0} and Lemma \ref{lem:DD1}, we have
\begin{equation*}
[\overline{f}](x^*, \xi^*) = [\overline{F_0}](x^*, \xi^*) + [\overline{F_1}](x^*, \xi^*) = I \exp \Bigl( -M_t(x^*, \xi^*; \tau_-(x^*, \xi^*) )\Bigr)
\end{equation*}
for $(x^*, \xi^*) \in \disc(\overline{f})$, which is the statement of Theorem \ref{thm:MR3}.

\section*{Acknowledgement}
The first author was supported in part by JSPS KAKENHI grant number 15K17572.

\end{document}